\documentclass[a4paper,12pt,final,leqno,notitlepage]{article}

\usepackage{amsmath}
\usepackage{amsfonts}
\usepackage{amssymb}
\usepackage{amsthm} 

\usepackage[blocks]{authblk}

\usepackage[T1]{fontenc}

\usepackage{fullpage}

\usepackage{enumerate}
\usepackage{enumitem}

\def\CC{\mathbb{C}} 
\def\DD{\mathbb{D}} 
\def\NN{\mathbb{N}} 
\def\CDD{\overline{\DD}} 
\def\OO{\mathcal{O}} 
\def\BBB{\mathcal{B}} 

\def\CL{\mathcal{C}} 

\def\INT{\mathrm{int}\,\,} 

\def\DIAM{\mathrm{diam}\,} 

\def\HULL#1#2{\widehat{#1}_{#2}} 
\def\HULLZ#1#2{\left(#1\right)\;\HULL{}{#2}} 
\def\HULLP#1{\HULL{#1}{}} 
\def\HULLZP#1{\HULLZ{#1}{}\,} 

\def\COMP#1#2{{#1}^{[#2]}} 

\def\UZW#1{\widetilde{#1}} 
\def\UZWI#1{\infty_{#1}} 

\def\HULLO#1{\HULL{#1}{\Omega}} 
\def\HULLZO#1{\HULLZ{#1}{\Omega}} 

\def\COMPPHI#1{\COMP{\varphi}{#1}} 
\def\COMPCPHI#1{\COMP{C_\varphi}{#1}} 

\def\UZWO{\UZW{\Omega}} 
\def\UZWIO{\UZWI{\Omega}} 

\newtheoremstyle{remarkstyle}{}{}{}{}{\bf}{.}{ }{}

\newtheorem{THE}{Theorem}[subsection]
\newtheorem{PROP}[THE]{Proposition}
\newtheorem{LEM}[THE]{Lemma}
\newtheorem{CRL}[THE]{Corollary}
\newtheorem{OBS}[THE]{Observation}

\theoremstyle{remarkstyle}
\newtheorem{RM}[THE]{Remark}
\newtheorem{EX}[THE]{Example}
\newtheorem{DEF}[THE]{Definition}
\newtheorem{PROBLEM}[THE]{Problem}

\begin{document}

\renewcommand{\thepage}{\small\arabic{page}}
\renewcommand{\thefootnote}{(\arabic{footnote})}
\renewcommand{\thesection}{\arabic{chapter}.\arabic{section}}
\renewcommand{\thesubsection}{\arabic{subsection}}

\renewcommand\Affilfont{\small}



\author{Sylwester Zaj\k{a}c}
\affil{Institute of Mathematics, Faculty of Mathematics and Computer Science,\\ Jagiellonian University, \L ojasiewicza 6, 30-348 Krak\'ow, Poland\\ sylwester.zajac@im.uj.edu.pl}

\title{Hypercyclicity of composition operators\\in Stein manifolds}

\date{}

\maketitle

\begin{abstract}
We characterise hypercyclic composition operators $C_\varphi:f\mapsto f\circ\varphi$ on the space of functions holomorphic on $\Omega$, where $\Omega$ is a connected Stein manifold and $\varphi$ is a holomorphic self-mapping of $\Omega$.

In the case when all balls with respect to the Carath\'{e}odory pseudodistance are relatively compact in $\Omega$, we show that much simpler characterisation is possible (many natural classes of domains in $\CC^N$ satisfy this condition).
Moreover, we show that in such a class of manifolds, and in simply connected and infinitely connected planar domains, hypercyclicity of $C_\varphi$ implies its hereditary hypercyclicity. \end{abstract}

\footnotetext[1]{


{\bf Acknowledgement:}
Main part of this paper was prepared as a semester paper under the guidance of P. Doma\'{n}ski in the framework of joint Ph. D.
program \'{S}SDNM (Poland) during the author research stay at Faculty of Mathematics and Computer Science, Adam Mickiewicz University, Pozna\'{n} (Poland).

During the subsequent development of the paper the author was supported by the NCN grant on the basis of the decision number DEC-2012/05/N/ST1/02911.
}




\subsection{Introduction}\label{sect_intro}

Let $\Omega$ be a connected $N$-dimensional Stein manifold (in particular, $\Omega$ can be a domain of holomorphy in $\CC^N$) and let $\varphi:\Omega\to\Omega$ be a holomorphic mapping.
We are interested in the problem of hypercyclicity and hereditary hypercyclicity of the composition operator $C_\varphi:f\mapsto f\circ\varphi$ on the space $\OO(\Omega)$ of holomorphic functions $f:\Omega\to\CC$, endowed with the usual topology of locally uniform convergence.

In the case when $\Omega$ is a domain in $\CC$, a characterisation of hypercyclicity was given by Grosse-Erdmann and Mortini in \cite{grosseerdmann_mortini} (actually, they described universal sequences $(C_{\varphi_n})_{n\in\NN}$, where $\varphi_n:\Omega\to\Omega$ are holomorphic maps).
In higher dimensions the problem was considered by several authors, mostly in cases when $\Omega$ is a polydisc, an euclidean ball or the whole $\CC^N$ with $\varphi$ being special (see \cite{bernalgonzalez} and the references in \cite{grosseerdmann_mortini}).
Analogous problem in spaces of real analytic functions was considered in \cite{domanskibonet} and in kernels of general (non-necessarily Cauchy-Riemann) partial differential equations in \cite{kalmesniess}.

From \cite[Theorem 3.21]{grosseerdmann_mortini} it follows (after applying a short reasoning; see Section \ref{sect_one_dim} in this paper) that if $\Omega$ is a simply connected or an infinitely connected domain in $\CC$, then for a holomorphic map $\varphi:\Omega\to\Omega$ the operator $C_\varphi$ is hypercyclic if and only if $\varphi$ is injective, $\varphi(\Omega)$ is a Runge domain with respect to $\Omega$ and the sequence $(\COMPPHI{n})_n$ is run-away, i.e. for every compact set $K\subset\Omega$ there is some $n$ such that $K\cap\COMPPHI{n}(K)=\varnothing$ ($\COMPPHI{n}$ is the $n$-th iterate of $\varphi$; see Section \ref{sect_preliminaries} for necessary definitions).
Moreover, the same theorem states that if $\Omega\subset\CC$ is finitely connected, then no $\varphi$ induces hypercyclic $C_\varphi$.
In Section \ref{sect_general} of this paper we characterise hypercyclic and hereditarily hypercyclic operators $C_\varphi$ for an arbitrary connected $N$-dimensional Stein manifold $\Omega$ and an arbitrary holomorphic mapping $\varphi:\Omega\to\Omega$ (Theorems \ref{th_char} and \ref{th_char_2}), using ideas developed by several authors, e.g. \cite{grosseerdmann_mortini}, \cite{bernalgonzalez}.
We prove that if $\Omega$ is a connected $N$-dimensional Stein manifold and $\varphi:\Omega\to\Omega$ is a holomorphic map, then $C_\varphi$ is hypercyclic if and only if $\varphi$ is injective, $\varphi(\Omega)$ is a Runge domain and for every compact holomorphically convex (shortly: $\Omega$-convex) set $K\subset\Omega$ there is $n$ for which the sets $K$ and $\COMPPHI{n}(K)$ are disjoint and their sum is $\Omega$-convex.
The language of $\Omega$-convexity seems to be natural for working with hypercyclicity of $C_\varphi$ of the space $\OO(\Omega)$, as it is strongly connected with some approximation theorems for that space (e.g. the Runge and Oka-Weil theorems).
One of the reasons, for which in the case of dimension one it was possible to characterise hypercyclicity in topological terms, is that in that situation the Runge theorem makes it possible to translate the notion of $\Omega$-convexity to some topological properties, what is generally not possible in higher dimensions.
Making use of the characterisation, we formulate an interesting criterion for $C_\varphi$ to be hereditarily hypercyclic (Theorem \ref{the_carat_do_jedynki}): assuming that $\varphi$ is injective and its image is a Runge domain with respect to $\Omega$ (these assumptions are in fact necessary), if $c_\Omega(z_0,\COMPPHI{n}(z_0))\to \infty$ for some point $z_0\in\Omega$, then the operator $C_\varphi$ is hereditarily hypercyclic ($c_\Omega$ denotes the Carath\'{e}odory pseudodistance on $\Omega$).

It is a disadvantage of the conditions in our characterisation that they require to answer the question whether a sum of two disjoint holomorphically convex sets is holomorphically convex, while the description for simply and infinitely connected domains in $\CC$ avoid this problem.
This fact motivated us to ask about higher-dimensional $\Omega$'s for which that simplier description works.
Not every $\Omega$ has this property; the conditions which are equivalent for the classes of planar domains mentioned above, are in general necessary but not sufficient (even in the punctured disc $\DD\setminus\lbrace 0\rbrace$).
In Section \ref{sect_simp_char} we proved the following fact: if all balls with respect to the Carath\'{e}odory pseudodistance are relatively compact in the topology of $\Omega$, then $\Omega$ admits the same characterisation of hypercyclicity as the planar domains mentioned above (Theorem \ref{th_carat_conv}).
Such a class of $\Omega$'s includes many 'nice' domains in $\CC^N$, e.g. bounded convex domains, strictly pseudoconvex domains, analytic polyhedra, etc.

It follows from Theorem \ref{th_char_2} that the operator $C_\varphi$ is hypercyclic if and only if it is hereditarily hypercyclic with respect to some sequence $(n_l)_l$ (Observation \ref{obs_hu_subseq}; it was also noticed in \cite{grosseerdmann_mortini} for planar domains).
In Section \ref{sect_hu} we prove that in many $\Omega$'s even more is true: $C_\varphi$ is hypercyclic if and only if it is hereditarily hypercyclic (Theorem \ref{th_taut_and_conv_impl_hu}).
Again, it turns out that $N$-dimensional $\Omega$'s with relatively compact Carath\'{e}odory balls (Theorem \ref{th_carat_hu}) and simply connected and infinitely connected planar domains (Section \ref{sect_one_dim}) admit this property.
It remains an open question if every hypercyclic operator $C_\varphi$ is automatically hereditarily hypercyclic, when $\Omega$ is an arbitrary connected Stein manifold.

Actually, given an increasing sequence $(n_l)_l\in\NN$, in this paper we study the following problems:
\begin{enumerate}
\renewcommand{\theenumi}{(p\arabic{enumi})}
\renewcommand{\labelenumi}{\theenumi}
\item\label{probl_u} Hypercyclicity of $C_\varphi$.
\item\label{probl_u_nl} Hypercyclicity of $C_\varphi$ with respect to $(n_l)_l$.
\item\label{probl_hu_nl} Hereditary hypercyclicity of $C_\varphi$ with respect to $(n_l)_l$.
\item\label{probl_hu} Hereditary hypercyclicity of $C_\varphi$.
\end{enumerate}
For the definition of this notions see Definition \ref{def_hyp_2}.
It is clear that there hold the implications \ref{probl_hu} $\Rightarrow$ \ref{probl_hu_nl} $\Rightarrow$ \ref{probl_u_nl} $\Rightarrow$ \ref{probl_u}, and in Sections \ref{sect_hu} and \ref{sect_one_dim} we prove that \ref{probl_u} $\Rightarrow$ \ref{probl_hu} in some $\Omega$'s.

\subsection{Preliminaries}\label{sect_preliminaries}

In this paper $\DD$ denotes the open unit disc in $\CC$, $\DD_*=\DD\setminus\lbrace 0\rbrace$ and $\CC_*=\CC\setminus\lbrace 0\rbrace$.

Throughout this section let $\Omega, \Omega'$ be connected finite-dimensional complex analytic manifolds.

By $\OO(\Omega,\Omega')$ we denote the set of all holomorphic mappings $f:\Omega\to\Omega'$.
In the case $\Omega'=\CC^M$ we equip the linear space $\OO(\Omega,\CC^M)$ with the compact-open topology, i.e. the topology of uniform convergence on compact subsets.
In the case $M=1$ we shortly write $\OO(\Omega)$ instead of $\OO(\Omega,\CC)$.

By $\UZWO=\Omega\cup\lbrace\UZWIO\rbrace$ we denote the usual compactification of (locally compact topological space) $\Omega$ by an element $\UZWIO\not\in\Omega$.

We say that a sequence $(K_l)_{l\in\NN}$ of compact subsets of $\Omega$ is an \textit{exhaustion} of $\Omega$ if $\bigcup_l K_l=\Omega$ and $K_l\subset\INT K_{l+1}$.
We say that $\Omega$ is \textit{countable at infinity} if there exists an exhaustion of $\Omega$.

We say that a sequence of holomorphic functions $f_n:\Omega\to\Omega'$ is \textit{compactly divergent} (\textit{in} $\OO(\Omega,\Omega')$) if for each compact subsets $K\subset\Omega$, $L\subset\Omega'$ there is $n_0$ such that $f_n(K)\cap L=\varnothing$ for all $n\geq n_0$.
We say that$(f_n)_n$ is \textit{run-away} (\textit{in} $\OO(\Omega,\Omega')$) if for each compact subsets $K\subset\Omega, L\subset\Omega'$ there is $n$ such that $f_n(K)\cap L=\varnothing$.
In the case $\Omega=\Omega'$ it is always enough to consider the situation when $L=K$.
Note that in the case when $\Omega$ and $\Omega'$ are countable at infinity, the sequence $(f_n)_n$ is run-away if and only if it has a compactly divergent subsequence, and $(f_n)_n$ is compactly divergent if and only if each of its subsequences is run-away.

We say that a holomorphic map $f:\Omega\to\Omega'$ is \textit{regular} if its derivative is a monomorphism at each point of $\Omega$.
We say that $f$ is \textit{almost proper} if for every compact set $K\subset\Omega'$, each connected component of $f^{-1}(K)$ is compact.
We say that $f$ is \textit{proper} if $f^{-1}(K)$ is compact for every compact set $K\subset\Omega'$.

For points $z,w\in\Omega$ let
\begin{eqnarray*}
c_\Omega^*(z,w) &:=& \sup \lbrace |F(z)|: F\in\OO(\Omega,\DD), F(w)=0\rbrace,\\
c_\Omega(z,w) &:=& \frac12 \log\frac{1+c_\Omega^*(z,w)}{1-c_\Omega^*(z,w)}.
\end{eqnarray*}
Here $c_\Omega$ is the Carath\'{e}odory pseudodistance and $c_\Omega^*$ is the M\"{o}bius pseudodistance in $\Omega$.
For more informations we refer the reader to \cite{jarnickipflug}.
In Sections \ref{sect_simp_char} and \ref{sect_hu} we deal with (connected) Stein manifolds for which all balls w.r.t. $c_\Omega$ are relatively compact in the topology of $\Omega$, showing that they have some good hypercyclicity properties.

Given a set $X$, a mapping $T:X\to X$ and an integer number $n$, we denote by $\COMP{T}{n}$ the $n$-th iteration of $T$, i.e. the mapping $T\circ T\circ\ldots\circ T$ ($n$ times).

Any topological vector spaces in this paper are assumed to lie over the filed $\CC$.

\begin{DEF}\label{def_hyp_1}
Let $T_n$ ($n\in\NN$) be continuous self-maps of a topological space $X$.
\begin{enumerate}
\renewcommand{\theenumi}{(\arabic{enumi})}
\renewcommand{\labelenumi}{\theenumi}
\item We say that the sequence $(T_n)_n$ is \textit{topologically transitive} if for every non-empty open subsets $U, V\subset X$ there exists $n$ such that $T_n(U)\cap V\neq\varnothing$.
\item We call a point $x\in X$ an \textit{universal element} for $(T_n)_n$ if the set $\lbrace T_n(x):n\in\NN\rbrace$ is dense in $X$.
We say that the sequence $(T_n)_n$ is \textit{universal} if it admits a universal element.
We say that $(T_n)_n$ is \textit{hereditarily universal} if each of its subsequences is universal.
\end{enumerate}
\end{DEF}

\begin{DEF}\label{def_hyp_2}
Let $T$ be a continuous linear operator on a topological vector space $X$.
We say that the mapping $T$ is \textit{hypercyclic} (resp. \textit{hereditarily hypercyclic}) with respect to an increasing sequence $(n_l)_l\subset\NN$ if the sequence $(\COMP{T}{n_l})_l$ is universal (resp. hereditarily universal).
We call $T$ shortly \textit{hypercyclic} (resp. \textit{hereditarily hypercyclic}) if it is hypercyclic (resp. hereditarily hypercyclic) w.r.t. the full sequence $(n)_n$.
If $T$ is hypercyclic, then any universal element of $(\COMP{T}{n})_n$ is called a hypercyclic vector.
\end{DEF}

Let us recall two classical theorems which are essential for our considerations.
For the first one see e.g. \cite[Theorem 1]{grosseerdmann_families}).
The second is due to A. Peris (see \cite[Proposition 1]{grosseerdmann_families}).

\begin{THE}\label{th_birkhoff}
Let $X$ be a separable Fr\'{e}chet space.
A sequence $(T_n)_n$ of continuous self-maps of $X$ is topologically transitive if and only if the set of its universal elements is dense in $X$.

Moreover, if one of these conditions holds, then the set of universal elements for $(T_n)_n$ is a dense $G_\delta$-subset of $X$.
\end{THE}

\begin{THE}\label{th_empty_or_dense}
Let $X$ be a separable Fr\'{e}chet space.
Suppose that $(T_n)_n$ is a sequence of continuous self-maps of $X$ such that each $T_n$ has dense range and that the family $(T_n)_n$ is commuting, i.e. $$T_n\circ T_m = T_m\circ T_n,\,\textnormal{ for }m, n\in\NN.$$
Then the set of universal elements for $(T_n)_n$ is empty or dense.
\end{THE}

\noindent
From these theorems there follows an immediate corollary.
It plays a key role in deriving the equivalent conditions for hypercyclicity and hereditary hypercyclicity (Section \ref{sect_general}), because we shall often investigate topological transitivity instead of hypercyclicity.
A similar argument was used in \cite{grosseerdmann_mortini}.

\begin{CRL}\label{cor_hc_tt}
Let $X$ be a separable Fr\'{e}chet space, let $T:X\to X$ be a continuous map, and let $(n_l)_l\subset\NN$ be an increasing sequence.
Then $T$ is hypercyclic w.r.t. $(n_l)_l$ if and only if the sequence $(\COMP{T}{n_l})_l$ is topologically transitive.
\end{CRL}

Now, let us recall the notion of holomorphic convexity and some notions connected with it.
They are extensively used in the paper.
For a compact set $K\subset\Omega$ by $\HULLO{K}$ or $\HULLZO{K}$ we denote the \textit{holomorphic hull} of the set $K$ with respect to $\Omega$, i.e. $$\HULLO{K}:=\lbrace z\in\Omega: |f(z)|\leq\sup_K |f|\textnormal{ for every }f\in\OO(\Omega) \rbrace.$$
The set $K$ is called \textit{holomorphically convex} if $K=\HULLO{K}$; we call such a set shortly: $\Omega$\textit{-convex}.
For $\Omega=\CC^N$ we shortly write $\HULLP{K}$ or $\HULLZP{K}$, we call these sets \textit{polynomial hulls}, and we say that $K$ is \textit{polynomially convex} if $K=\HULLP{K}$.

We call a domain $U\subset\Omega$ a Runge domain with respect to $\Omega$, if every function from $\OO(U)$ can be approximated locally uniformly on $U$ by functions from $\OO(\Omega)$.
In \cite{hormander} the reader can find many informations about holomorphic hulls and Runge domains.
Neverthless, below we recall these theorems which are important for our purposes.

Let us introduce a special class of complex analytic manifolds, called Stein manifolds.
Many informations may be found in \cite[Chapter 5]{hormander}.
The reader who is not familiar with Stein manifolds, for the rest of the paper may assume that $\Omega$ is a domain of holomorphy in $\CC^N$.	

\begin{DEF}\label{def_stein}
A complex analytic manifold $\Omega$ of (finite) dimension $N$ is said to be a Stein manifold, if:
\begin{enumerate}
\renewcommand{\theenumi}{(\arabic{enumi})}
\renewcommand{\labelenumi}{\theenumi}
\item\label{def_stein_1} $\Omega$ is countable at infinity,
\item\label{def_stein_2} the set $\HULLO{K}$ is compact for every compact subset $K$ of $\Omega$,
\item\label{def_stein_3} the family $\OO(\Omega)$ separates points in $\Omega$, i.e. for each $z,w\in\Omega$, $z\neq w$, there exists $f\in\OO(\Omega)$ with $f(z)\neq f(w)$,
\item\label{def_stein_4} for any $z\in\Omega$ there is a map $F\in\OO(\Omega,\CC^N)$ which forms a local coordinate system at $z$, i.e. the derivative of $F$ at $z$ is an isomorphism.
\end{enumerate}
\end{DEF}

\noindent

The assumption that $\Omega$ is countable at infinity guarantees that the space $\OO(\Omega)$ is a Fr\'{e}chet space, because its topology is given by a countable family of seminorms $p_l:f\mapsto\sup_{K_l}|f|$, $l\in\NN$.
Moreover, this assumption implies that $\OO(\Omega)$ is separable, because the space $\CL(\Omega)$ (endowed with the same topology) is so.
This observation allows us to use Corollary \ref{cor_hc_tt} for the space $X=\OO(\Omega)$, with $\Omega$ being a connected Stein manifold.
But in fact, the main reason for which we work with Stein manifolds is the following theorem (see \cite[Corollary 5.2.9]{hormander}):

\begin{THE}[Oka-Weil]\label{the_oka_weil}
Let $\Omega$ be a Stein manifold and let $K\subset\Omega$ be a compact $\Omega$-convex set.
Then every function which is holomorphic in a neighborhood of $K$ can be approximated uniformly on $K$ by functions from $\OO(\Omega)$.
\end{THE}

The Oka-Weil theorem turns out to be a very good tool for showing topological transitivity of sequence $(\COMPCPHI{n_l})_l$, because - as we shall see - the question of topological transitivity of that sequence may be translated to a question of approximation of some functions by elements of $\OO(\Omega)$.
Its one-dimensional version, known as Runge theorem, was used in \cite{grosseerdmann_mortini}.

The following well-known fact characterises Runge domains in a Stein manifold $\Omega$ in the language of holomorphic hulls:

\begin{THE}\label{th_runge_domain_characterisation}
Let $\Omega$ be a connected Stein manifold and let $U\subset\Omega$ be a domain which is also a Stein manifold.
Then the following conditions are equivalent:
\begin{enumerate}
\renewcommand{\theenumi}{(\arabic{enumi})}
\renewcommand{\labelenumi}{\theenumi}
\item The domain $U$ is a Runge domain in $\Omega$.
\item For every compact subset $K\subset U$ we have $\HULLO{K}=\HULL{K}{U}$.
\item For every compact subset $K\subset U$ we have $\HULLO{K}\cap U=\HULL{K}{U}$.
\item For every compact subset $K\subset U$ we have $\HULLO{K}\cap U\subset\subset U$.
\end{enumerate}
\end{THE}

For the proof in the case when $\Omega$ is a domain of holomorphy in $\CC^N$, see e.g. the proof of Theorem 4.3.3 in \cite{hormander}.
The proof in the case of Stein manifold is actually the same; it only uses \cite[Corollary 5.2.9]{hormander} instead of \cite[Theorem 4.3.2]{hormander}.

To simplify notation in this article, we introduce the following notion: we say that compact sets $K,L\subset\Omega$ are \textit{separable in} $\Omega$ ($\Omega$ is a connected Stein manifold), if there exists a function $F\in\OO(\Omega)$ which \textit{separates} $K$ and $L$, i.e. $$\HULLP{F(K)}\cap\HULLP{F(L)} = \varnothing.$$
In this situation there holds $$\HULLO{K}\cap\HULLO{L}=\varnothing,$$ because $\HULLO{K}\subset F^{-1}(\HULLP{F(K)})$, $\HULLO{L}\subset F^{-1}(\HULLP{F(L)})$.

In the lemmas below we present some natural properties of holomorphic hull of a sum of two compact subsets.
For the case of polynomial hull they can be found e.g. in \cite{stout}(Theorem 1.6.19 and Corollary 1.5.4).
However, we were not able to find them in the form as we need, so for the reader's convenience we present sketches of proofs.

\begin{LEM}\label{lemat_otoczki_1}
Let $\Omega$ be a connected Stein manifold and let $K, L\subset\Omega$ be compact subsets.
Then the following conditions are equivalent:
\begin{enumerate}
\renewcommand{\theenumi}{(\arabic{enumi})}
\renewcommand{\labelenumi}{\theenumi}
\item\label{lem_ot_funkcja} The sets $K, L$ are separable in $\Omega$.
\item\label{lem_ot_rozklad} There exist open and disjoint subsets $U, V\subset\Omega$ such that $\HULLO{K}\subset U$, $\HULLO{L}\subset V$ and $\HULLZO{K\cup L}\subset U\cup V$.
\item\label{lem_ot_suma} $\HULLO{K}\cap\HULLO{L}=\varnothing$ and $\HULLZO{K\cup L}=\HULLO{K}\cup\HULLO{L}$.
\end{enumerate}

In particular, if $K$ and $L$ are disjoint and $\Omega$-convex, then $K\cup L$ is $\Omega$-convex if and only if $K$ and $L$ are separable in $\Omega$.
\end{LEM}

\begin{proof}[Sketch of the proof]
\ref{lem_ot_suma} $\Rightarrow$ \ref{lem_ot_funkcja}: Take a function $f$ equal $0$ in a neighborhood of $\HULLO{K}$ and $1$ in a neighborhood of $\HULLO{L}$.
By \ref{lem_ot_suma}, this function is holomorphic in a neighborhood of the $\Omega$-convex set $\HULLZO{K\cup L}$, so in view of the Oka-Weil theorem it can be approximated on this set by functions holomorphic on $\Omega$.
Hence there is some $F\in\OO(\Omega)$ such that $|F-f|<\frac12$ on $\HULLZO{K\cup L}$.
This $F$ satisfies \ref{lem_ot_funkcja}.

\ref{lem_ot_funkcja} $\Rightarrow$ \ref{lem_ot_rozklad}: Define $U:=F^{-1}(U_0), V:=F^{-1}(V_0)$, where $U_0$ and $V_0$ are some disjoint open neighborhoods of the compact sets $\HULLZP{F(K)}$ and $\HULLZP{F(L)}$, respectively.
The condition $\HULLZP{F(K)}\cap\HULLZP{F(L)} = \varnothing$ implies that $\HULLZP{F(K)\cup F(L)} = \HULLZP{F(K)}\cup\HULLZP{F(L)}$ (this equality holds for polynomial hulls on the complex plane; see \cite[Theorem 1.3.3]{hormander}).
Therefore $\HULLZO{K\cup L}\subset F^{-1}\left(\HULLZP{F(K\cup L)}\right)\subset U\cup V$.

\ref{lem_ot_rozklad} $\Rightarrow$ \ref{lem_ot_suma}: The right-to-left inclusion is obvious, so we prove the other one.
Fix $z_0\in I:=\HULLZO{K\cup L}$.
We can assume that $z_0\in U$.
We prove that $z_0\in\HULLO{K}$.
The characteristic function $\chi_U$ of $U$, restricted to the set $U\cup V$, is in view of \ref{lem_ot_rozklad} holomorphic in a neighborhood of the $\Omega$-convex set $I$, so there exists a sequence of functions $(g_n)_n\subset\OO(\Omega)$ uniformly convergent to $\chi_U$ on $I$.
For any function $f\in\OO(\Omega)$ the sequence $(fg_n)_n$ converges uniformly to $f\chi_U$ on $I$, so
$$|f(z_0)| = \lim_{n\to\infty}|f(z_0)g_n(z_0)| \leq \lim_{n\to\infty}\sup_{z \in K\cup L}|f(z)g_n(z)|=\sup_{z\in K}|f(z)|.$$
This implies that $z_0\in\HULLO{K}$ and finishes the proof.
\end{proof}

\begin{LEM}\label{lemat_otoczki_2}
Let $\Omega$ be a connected Stein manifold and let $K,L\subset\Omega$ be compact subsets of $\Omega$.
If $K\cap L=\varnothing$ and the set $K\cup L$ is $\Omega$-convex, then $K$ and $L$ are both $\Omega$-convex.
\end{LEM}

\begin{proof}
As above, we can find a function $F\in\OO(\Omega)$ such that $|F|<\frac12$ on $K$ and $|F-1|<\frac12$ on $L$.
There is $\HULLO{K}\subset K\cup L$ and
$$\HULLO{K}\cap L\subset F^{-1}\left(\HULLP{F(K)}\right)\cap F^{-1}(F(L))\subset F^{-1}\left(\frac12\DD\right)\cap F^{-1}\left(1+\frac12\DD\right)=\varnothing,$$
so $\HULLO{K}\subset K$.
\end{proof}

\subsection{General results}\label{sect_general}

We start this section with formulating some necessary conditions.
They in fact appear in several papers.

\begin{PROP}\label{prop_kon_u}
Let $\Omega$ be a connected Stein manifold, $\varphi\in\OO(\Omega,\Omega)$ and let $(n_l)_l\in\NN$ be an increasing sequence.
Suppose that $C_\varphi$ is hypercyclic w.r.t. $(n_l)_l$.
Then:
\begin{enumerate}
\renewcommand{\theenumi}{(c\arabic{enumi})}
\renewcommand{\labelenumi}{\theenumi}
\item\label{pkon_u_1} The mapping $\varphi$ is injective.
\item\label{pkon_u_2} The image $\varphi(\Omega)$ is a Runge domain w.r.t. $\Omega$.
\item\label{pkon_u_3} The sequence $(\COMPPHI{n_l})_l$ is run-away.
\end{enumerate}
\end{PROP}

\begin{RM}\label{rem_c1_c2_c3}
Note that the first condition gives that $\varphi$ is a biholomorphism on its image, which is a classical result (see e.g. \cite[Theorem 2.2.1]{jakobczakjarnicki}); hence, $\varphi(\Omega)$ is a Stein manifold.
In view of Theorem \ref{th_runge_domain_characterisation}, the conditions \ref{pkon_u_1} and \ref{pkon_u_2} imply that the set $\varphi(K)$ is $\Omega$-convex for each compact $\Omega$-convex subset $K\subset\Omega$.
This immediately implies that for any integer number $n$ the set $\COMPPHI{n}(K)$ also is $\Omega$-convex.
\end{RM}

\begin{proof}[Proof of Proposition \ref{prop_kon_u}]
Let $f$ be a universal element for $(\COMPCPHI{n_l})_l$.
The first part follows from the condition \ref{def_stein_3} in Definition \ref{def_stein}.
For the second we need to prove that the restrictions $g|_{\varphi(\Omega)}, g\in\OO(\Omega)$, are dense in $\OO(\varphi(\Omega))$.
If $h\in\OO(\varphi(\Omega))$, then $h\circ\varphi$ is holomorphic on $\Omega$, so there is a sequence $(l_k)_k$ such that $f\circ\COMPPHI{n_{l_k}}\to h\circ\varphi$ on $\Omega$.
Hence $f\circ\COMPPHI{n_{l_k}-1}\to h$ on $\varphi(\Omega)$, as the mapping $\varphi$ is a biholomorphism on its image.

We prove the third part.
Let $K\subset\Omega$ be compact.
For each $j\in\NN$ there exists $l_j$ such that $|f\circ\COMPPHI{n_{l_j}}-j|\leq\frac{1}{j}$ on $K$.
This implies $$\inf_{z\in\COMPPHI{n_{l_j}}(K)}|f(z)| = \inf_{z\in K}|f\circ\COMPPHI{n_{l_j}}(z)|\geq j-\frac{1}{j}>\sup_{z \in K}|f(z)|\;\textnormal{ for big }j,$$
so $\COMPPHI{n_{l_j}}(K)\cap K=\varnothing$.
\end{proof}

Necessary conditions for hereditary hypercyclicity are similar:

\begin{PROP}\label{prop_kon_hu}
Let $\Omega$ be a connected Stein manifold, $\varphi\in\OO(\Omega,\Omega)$ and let $(n_l)_l\in\NN$ be an increasing sequence.
Suppose that $C_\varphi$ is hereditarily hypercyclic w.r.t. $(n_l)_l$.
Then:
\begin{enumerate}
\renewcommand{\theenumi}{(h\arabic{enumi})}
\renewcommand{\labelenumi}{\theenumi}
\item\label{pkon_hu_1} The mapping $\varphi$ is injective.
\item\label{pkon_hu_2} The image $\varphi(\Omega)$ is a Runge domain w.r.t. $\Omega$.
\item\label{pkon_hu_3} The sequence $(\COMPPHI{n_l})_l$ is compactly divergent.
\end{enumerate}
\end{PROP}

\begin{proof}
The first two parts follow from the previous proposition.
For the last one it is enough to recall that a sequence of holomorphic mappings is compactly divergent if and only if each of its subsequences is run-away.
\end{proof}

It is natural to ask whether the necessary conditions given by Propositions \ref{prop_kon_u} and \ref{prop_kon_hu} are sufficient.
In \cite{grosseerdmann_mortini} it is shown that this is true if $\Omega$ is a simple connected or an infinitely connected planar domain (see also Section \ref{sect_one_dim} in this paper).
In Section \ref{sect_simp_char} we give a class of higher-dimensional Stein manifolds where this also holds.
But in general the above necessary conditions are not sufficient, as we can see using a simple example $\Omega=\DD_*$ and $\varphi(z)=\frac12 z$ - then by Theorems \ref{th_char} or \ref{th_char_w_c00} the operator $C_\varphi$ is not hypercyclic, although it satisfies the conditions \ref{pkon_u_1}, \ref{pkon_u_2}, \ref{pkon_u_3}.

We are going to give equivalent conditions for hypercyclicity and hereditary hypercyclicity of $C_\varphi$.
We use Corollary \ref{cor_hc_tt}, so let us first say what topological transitivity of the sequence $(\COMPCPHI{n_l})_l$ means.
Let $\varphi:\Omega\to\Omega$ be an injective holomorphic mapping.
The sets $$W_{f_0,K,\epsilon}:=\lbrace f\in\OO(\Omega): |f-f_0|<\epsilon\textnormal{ on }K\rbrace,\;f_0\in\OO(\Omega),\epsilon>0,K\subset\Omega\textnormal{ compact},$$ form a basis of the topology of $\OO(\Omega)$.
Thus, the sequence $(\COMPCPHI{n_l})_l$ is topologically transitive if and only if for every $\epsilon>0$, $g, h$ holomorphic on $\Omega$ and compact $K\subset\Omega$ there are $l\in\NN$ and a function $f$ holomorphic on $\Omega$ such that
\begin{align*}
|f-g|<\epsilon\textnormal{ on }K\textnormal{ and }|f\circ\COMPPHI{n_l}-h|<\epsilon\textnormal{ on }K.
\end{align*}
As the mapping $\varphi$ is injective, the above condition takes the form:
\begin{align*}\label{warunek_birkhoff}
\tag{T} |f-g|<\epsilon\textnormal{ on }K\textnormal{ and }|f-h\circ\COMPPHI{-n_l}|<\epsilon\textnormal{ on }\COMPPHI{n_l}(K).
\end{align*}
Note that in the case when $\Omega$ is a Stein manifold we can restrict to considering only $\Omega$-convex sets $K$.

\begin{THE}\label{th_char}
Let $\Omega$ be a connected Stein manifold, $\varphi\in\OO(\Omega,\Omega)$  and let $(n_l)_l\in\NN$ be an increasing sequence.
Then:
\begin{enumerate}
\renewcommand{\theenumi}{(\arabic{enumi})}
\renewcommand{\labelenumi}{\theenumi}
\item\label{th_char_u} The operator $C_\varphi$ is hypercyclic w.r.t. $(n_l)_l$ if and only if $\varphi$ is injective and for every $\Omega$-convex compact subset $K\subset\Omega$ there exists $l$ such that $K\cap\COMPPHI{n_l}(K)=\varnothing$ and the set $K\cup\COMPPHI{n_l}(K)$ is $\Omega$-convex.
\item\label{th_char_hu} The operator $C_\varphi$ is hereditarily hypercyclic w.r.t. $(n_l)_l$ if and only if $\varphi$ is injective and for every $\Omega$-convex compact subset $K\subset\Omega$ there exists $l_0$ such that $K\cap\COMPPHI{n_l}(K)=\varnothing$ and the set $K\cup\COMPPHI{n_l}(K)$ is $\Omega$-convex for each $l\geq l_0$.
\end{enumerate}
\end{THE}

\begin{proof}[Proof of Theorem \ref{th_char}]
First we prove \ref{th_char_u}.
Necessity.
Suppose that $C_\varphi$ is hypercyclic w.r.t. $(n_l)_l$.
In view of Corollary \ref{cor_hc_tt}, the condition (\ref{warunek_birkhoff}) holds.
Fix an $\Omega$-convex compact set $K\subset\Omega$.
By Remark \ref{rem_c1_c2_c3} we get that the set $\COMPPHI{n_l}(K)$ is $\Omega$-convex.
Using the condition (\ref{warunek_birkhoff}) for $g\equiv 0$, $h\equiv 1$, $\epsilon=\frac12$ we get that there are $f\in\OO(\Omega)$ and $l\in\NN$ such that $f(K)\subset\frac12\DD$ and $f(\COMPPHI{n_l}(K))\subset 1+\frac12\DD$.
This implies that $K$ and $\COMPPHI{n_l}(K)$ are separable in $\Omega$, so by Lemma \ref{lemat_otoczki_1} the sum $K\cup\COMPPHI{n_l}(K)$ is $\Omega$-convex.

Sufficiency.
We prove that the condition (\ref{warunek_birkhoff}) is satisfied.
Fix a compact $\Omega$-convex set $K\subset\Omega$, a number $\epsilon>0$ and holomorphic functions $g, h:\Omega\to\CC$.
Take $l$ such that $K\cap\COMPPHI{n_l}(K)=\varnothing$ and the set $I:=K\cup\COMPPHI{n_l}(K)$ is $\Omega$-convex.
The function $\widetilde{f}$ defined as $g$ in a neighborhood of $K$ and $h\circ\COMPPHI{-n_l}$ in a neighborhood of $\COMPPHI{n_l}(K)$ is well-defined and holomorphic in a neighborhood of an $\Omega$-convex set $I$, so in view of the Oka-Weil theorem it can be approximated by functions holomorphic on $\Omega$.
This implies that there exists a function $f\in\OO(\Omega)$ such that $|f-g|<\epsilon$ on $K$ and $|f-h\circ\COMPPHI{-n_l}|<\epsilon$ on $\COMPPHI{n_l}(K)$.

Observe, that \ref{th_char_hu} follows immediately from \ref{th_char_u}.
Indeed, the condition in \ref{th_char_hu} does not hold if and only if there is an $\Omega$-convex subset $K\subset\Omega$ and an increasing sequence $(l_k)_k$ such that for each $k$ the sets $K$ and $\COMPPHI{n_{l_k}}(K)$ are not disjoint or their sum is not $\Omega$-convex.
On the other hand, $C_\varphi$ is not hereditarily hypercyclic w.r.t. $(n_l)_l$ if and only if for some increasing sequence $(l_k)_k$ it is not hypercyclic w.r.t. $(n_{l_k})_k$.
In view of \ref{th_char_u}, these conditions are equivalent.
\end{proof}

\begin{THE}\label{th_char_2}
Let $\Omega$ be a connected Stein manifold, $\varphi\in\OO(\Omega,\Omega)$  and let $(n_l)_l\in\NN$ be an increasing sequence.
Then:
\begin{enumerate}
\renewcommand{\theenumi}{(\arabic{enumi})}
\renewcommand{\labelenumi}{\theenumi}
\item\label{th_char_2_u} The operator $C_\varphi$ is hypercyclic w.r.t. $(n_l)_l$ if and only if $\varphi$ is injective, $\varphi(\Omega)$ is a Runge domain w.r.t. $\Omega$ and for every $\Omega$-convex compact subset $K\subset\Omega$ there is $l\in\NN$ such that the sets $K$ and $\COMPPHI{n_l}(K)$ are separable in $\Omega$.
\item\label{th_char_2_hu} The operator $C_\varphi$ is hereditarily hypercyclic w.r.t. $(n_l)_l$  if and only if $\varphi$ is injective, $\varphi(\Omega)$ is a Runge domain w.r.t. $\Omega$ and for every $\Omega$-convex compact subset $K\subset\Omega$ there exists $l_0$ such that for each $l\geq l_0$ the sets $K$ and $\COMPPHI{n_l}(K)$ are separable in $\Omega$.
\end{enumerate}
\end{THE}

\begin{proof}
Sufficiency in both parts follows from Theorem \ref{th_char}: if the sets $K$ and $\COMPPHI{n_l}(K)$ are separable in $\Omega$, then by Lemma \ref{lemat_otoczki_1} their sum is $\Omega$-convex (since $\varphi(\Omega)$ is a Runge domain in $\Omega$, $\COMPPHI{n_l}(K)$ is $\Omega$-convex if $K$ is so).
Necessity in both parts follows directly from Theorem \ref{th_char}, Lemma \ref{lemat_otoczki_1} and the necessary conditions.
\end{proof}

Actually, from the above theorem it follows that (for an injective $\varphi$ with $\varphi(\Omega)$ being a Runge domain w.r.t. $\Omega$) to get hypercyclicity of $C_\varphi$ w.r.t. $(n_l)_l$ it suffices to prove the condition  (\ref{warunek_birkhoff}) for $g\equiv 0$, $h\equiv 1$ and $\epsilon=\frac{1}{2}$, i.e. to prove that for each compact $\Omega$-convex subset $K\subset\Omega$ there is a number $l$ and a function $f\in\OO(\Omega)$ such that 
$$|f|<\frac{1}{2}\textnormal{ on }K\textnormal{ and }|f-1|<\frac{1}{2}\textnormal{ on }\COMPPHI{n_l}(K).$$

\begin{OBS}\label{obs_hu_subseq}
Let $\Omega$ be a connected Stein manifold, $\varphi\in\OO(\Omega,\Omega)$  and let $(n_l)_l\in\NN$ be an increasing sequence.
Then the operator $C_\varphi$ is hypercyclic w.r.t. $(n_l)_l$ if and only if $(n_l)_l$  has a subsequence for which $C_\varphi$ is hereditarily hypercyclic.
\end{OBS}

In many 'nice' domains $\Omega$ there holds a stronger implication than the above: for any $\varphi\in\OO(\Omega,\Omega)$ hypercyclicity of $C_\varphi$ implies its hereditary hypercyclicity, i.e. the conditions \ref{probl_u} and \ref{probl_hu} are equivalent.
We deal with this topic in Section \ref{sect_hu}.

\begin{proof}
Indeed, assume that $C_\varphi$ is hypercyclic w.r.t. $(n_l)_l$ and let $(K_\mu)_\mu$ be a sequence of $\Omega$-convex compact sets which exhausts $\Omega$.
By Theorem \ref{th_char_2}, for each $\mu$ there is $l_\mu$ such that the sets $K_\mu$ and $\COMPPHI{n_{l_\mu}}(K_\mu)$ are separable in $\Omega$.
We may assume that the sequence $(l_\mu)_\mu$ is increasing.
We claim that $C_\varphi$ is hereditarily hypercyclic w.r.t. $(n_{l_\mu})_\mu$.
If $K\subset\Omega$ is compact, then there is $\mu_0$ such that $K$ is contained in $K_\mu$ for every $\mu\geq\mu_0$, what gives that the sets $K$ and $\COMPPHI{n_{l_\mu}}(K)$ are separable in $\Omega$.
\end{proof}

\begin{THE}\label{the_carat_do_jedynki}
Let $\Omega$ be a connected Stein manifold, $\varphi\in\OO(\Omega,\Omega)$ and let $(n_l)_l\in\NN$ be an increasing sequence.
Suppose that $\varphi$ is injective and that $\varphi(\Omega)$ is a Runge domain w.r.t. $\Omega$.
If there exists a point $z_0\in\Omega$ such that $$\lim_{l\to\infty}c_{\Omega}(z_0,\COMPPHI{n_l}(z_0))=\infty,$$
then the operator $C_\varphi$ is hereditarily hypercyclic w.r.t. $(n_l)_l$.
\end{THE}

The assumption $\lim_{l\to\infty}c_{\Omega}(z_0,\COMPPHI{n_l}(z_0))=\infty$ is fulfilled e.g. if there exists a function $F\in\OO(\Omega,\DD)$ such that $|F\circ\COMPPHI{n_l}(z_0)|\to 1 \textnormal{ as }l\to\infty$.

\begin{proof}
Since every subsequence of $(n_l)_l$ satisfy the same assumptions, it suffices to prove hypercyclicity.
The limit condition in the assumptions means that there exists a sequence $(F_l)_l\subset\OO(\Omega,\DD)$ such that $F_l(z_0)=0$ and $F_l(\COMPPHI{n_l}(z_0))\to 1$ as $l\to\infty$.
Using the Montel theorem and passing to a subsequence we may assume that $F_l\circ\COMPPHI{n_l}\to G$ and $F_l\to H$ for some functions $G,H\in\OO(\Omega)$ with $G(\Omega), H(\Omega)\subset\CDD$.
Since $G(z_0)=1$ and $H(z_0)=0$, if follows from the maximum principle that $G\equiv 1$ and $H(\Omega)\subset\DD$.

Fix a compact $\Omega$-convex subset $K\subset\Omega$.
There is an $\alpha\in(0,1)$ so that $H(K)\subset\alpha\DD$, so for big $l$ there holds $F_l(K)\subset\alpha\DD$.
On the other hand, $F_l(\COMPPHI{n_l}(K))\subset 1+(1-\alpha)\DD$, because $F_l\circ\COMPPHI{n_l}\to 1$.
This implies that for big $l$ the function $F_l$ separates $K$ and $\COMPPHI{n_l}(K)$.
Now apply Theorem \ref{th_char_2}.
\end{proof}

For a natural number $M$ introduce the operator $$C_{\varphi, M}:\OO(\Omega,\CC^M)\ni f\mapsto f\circ\varphi\in\OO(\Omega,\CC^M).$$
It turns out that if $C_\varphi$ is hypercyclic w.r.t. $(n_l)_l$, then every $C_{\varphi, M}$ is hypercyclic w.r.t $(n_l)_l$ (Corollary \ref{cor_C_fi_M}).
But here, it is interesting that making use of some classic results on Stein manifolds, we can show that for $M\geq 2N+1$ most (i.e. a dense $\mathcal{G}_\delta$ subset) of universal elements of the sequence $(\COMP{C_{\varphi,M}}{n_l})_l$ have nice properties: they are injective, regular, almost proper holomorphic maps from $\Omega$ to $\CC^M$ (Corollary \ref{cor_C_fi_M_universal_elements}).

Obviously, if for some $M$ the operator $C_{\varphi,M}$ is hypercyclic (resp. hereditarily hypercyclic) w.r.t. $(n_l)_l$, then $C_\varphi$ is hypercyclic (resp. hereditarily hypercyclic) w.r.t. $(n_l)_l$, as any universal element $f=(f_1,\ldots,f_M)$ for $(\COMP{C_{\varphi,M}}{n_l})_l$ gives universal elements $f_1,\ldots,f_M$ for $(\COMPCPHI{n_l})_l$.

Note that in general a proper map $f:\Omega\to\CC^M$ cannot be a universal element for $(\COMP{C_{\varphi,M}}{n_l})_l$.
For example, take $\Omega$ and $\varphi$ such that $C_\varphi$ is hypercyclic and the sequence $(\COMPPHI{n})_n$ is compactly divergent in $\OO(\Omega,\Omega)$.
If $f:\Omega\to\CC^M$ is holomorphic and proper, then $f\circ\COMPPHI{n}\to\infty$ compactly uniformly on $\Omega$, i.e. $\inf_{K}|f\circ\COMPPHI{n}|\to\infty$ for each compact set $K\subset\Omega$, so constant mappings cannot be approximated by maps of the form $f\circ\COMPPHI{n}$.

\begin{CRL}\label{cor_C_fi_M}
Let $\Omega$ be a connected Stein manifold, $\varphi\in\OO(\Omega,\Omega)$ and let $(n_l)_l\in\NN$ be an increasing sequence.

If $C_\varphi$ is hypercyclic (resp. hereditarily hypercyclic) w.r.t. $(n_l)_l$, then $C_{\varphi,M}$ is hypercyclic (resp. hereditarily hypercyclic) w.r.t. $(n_l)_l$ for every $M$.
\end{CRL}

\begin{proof}
It is enough to consider the case of hypercyclicity.
Obviously $\varphi$ is injective.
In view of Corollary \ref{cor_hc_tt}, it suffices to show that the sequence $(\COMP{C_{\varphi, M}}{n_l})_l)$ is topologically transitive.

Fix a number $\epsilon>0$, a compact $\Omega$-convex subset $K\subset\Omega$ and functions $g_1,\ldots,g_M$, $h_1,\ldots,h_M\in\OO(\Omega)$.
By Theorem \ref{th_char}, there is $l\in\NN$ such that the sets $K$ and $\COMPPHI{n_l}(K)$ are disjoint and their sum is $\Omega$-convex.
We need to show that there exist functions $f_1,\ldots,f_M\in\OO(\Omega)$ such that for $j=1,\ldots,M$ there is
$$|f_j-g_j|<\epsilon\textnormal{ on }K\textnormal{ and }|f_j-h_j\circ\COMPPHI{-n_l}|<\epsilon\textnormal{ on }\COMPPHI{n_l}(K).$$
But, as in the proof of Theorem \ref{th_char}, this follows from the fact that the functions $\widetilde{f_j}$ defined as $g_j$ in a neighborhood of $K$ and as $h_j\circ\COMPPHI{-n_l}$ in a neighborhood of $\COMPPHI{n_l}(K)$ can be approximated on the $\Omega$-convex set $K\cup\COMPPHI{n_l}(K)$ by functions holomorphic on $\Omega$.
\end{proof}

\begin{CRL}\label{cor_C_fi_M_universal_elements}
Let $\Omega$ be a connected $N$-dimensional Stein manifold, $\varphi\in\OO(\Omega,\Omega)$ and let $(n_l)_l\in\NN$ be an increasing sequence.
Assume that $C_\varphi$ is hypercyclic w.r.t. $(n_l)_l$.
Then:
\begin{enumerate}
\renewcommand{\theenumi}{(\arabic{enumi})}
\renewcommand{\labelenumi}{\theenumi}
\item For $M\geq N$, the set of all almost proper holomorphic maps $f:\Omega\to\CC^M$, which are universal elements for $(\COMP{C_{\varphi,M}}{n_l})_l$ contains a dense $\mathcal{G}_\delta$ subset of $\OO(\Omega,\CC^M)$.
\item For $M\geq 2N$, the set of all regular almost proper holomorphic maps $f:\Omega\to\CC^M$, which are universal elements for $(\COMP{C_{\varphi,M}}{n_l})_l$ contains a dense $\mathcal{G}_\delta$ subset of $\OO(\Omega,\CC^M)$.
\item For $M\geq 2N+1$, the set of all injective regular almost proper holomorphic maps $f:\Omega\to\CC^M$, which are universal elements for $(\COMP{C_{\varphi,M}}{n_l})_l$ contains a dense $\mathcal{G}_\delta$ subset of $\OO(\Omega,\CC^M)$.
\end{enumerate}
\end{CRL}

\begin{proof}
In view of Corollary \ref{cor_C_fi_M}, every $C_{\varphi,M}$ is hypercyclic w.r.t. $(n_l)_l$.
By \cite[Theorem 8.1.1]{forstneric}, if $M\geq N$, then the set of all almost proper holomorphic maps $f:\Omega\to\CC^M$ contains a dense $\mathcal{G}_\delta$ set.
By \cite[Theorem 5.3.6]{hormander}, if $M\geq 2N$ (resp. $M\geq 2N+1$), then the set of all regular holomorphic maps (resp. of all injective regular holomorphic maps) $f:\Omega\to\CC^M$ contains a dense $\mathcal{G}_\delta$ set.
On the other hand, by Theorem \ref{th_birkhoff} and Corollary \ref{cor_hc_tt}, the set of all universal elements for $(\COMP{C_{\varphi,M}}{n_l})_l$ is a dense $\mathcal{G}_\delta$ set.
An intersection of finitely many dense $\mathcal{G}_\delta$ subsets of $\OO(\Omega,\CC^M)$ is a dense $\mathcal{G}_\delta$ set, in view of the Baire theorem.
\end{proof}

Given a sequence $(\varphi_l)_l$ of holomorphic self-maps of a domain $\Omega\subset\CC^N$, we say that a sequence $(C_{\varphi_l})_l$ is $\BBB$-universal if the sequence $(C_{\varphi_l}|_{\BBB(\Omega)})_l$ is universal as a sequence of mappings from $\BBB(\Omega)$ to $\BBB(\Omega)$, where $\BBB(\Omega):=\OO(\Omega,\DD)$, endowed with the topology induced from $\OO(\Omega)$.
In \cite[Theorem 3.5 and Corollary 3.7]{gorkin} and \cite[Theorem 2.4]{grosseerdmann_mortini} the authors proved that, under certain assumptions (e.g. that bounded analytic functions are dense in $\OO(\Omega)$), $\BBB$-universality of sequences of composition operators $(C_{\varphi_l})_l$ implies its universality, where $\varphi_l$ are holomorphic self-maps of $\Omega$.
Restricting to the case of iterations, i.e. $\varphi_l=\COMPPHI{n_l}$, and to $\Omega$ being a domain of holomorphy in $\CC^N$, we can state a similar conclusion with less assumptions:

\begin{CRL}
Let $\Omega\subset\CC^N$ be a domain of holomorphy, $\varphi\in\OO(\Omega,\Omega)$ and let $(n_l)_l\in\NN$ be an increasing sequence.
Suppose that $\varphi$ is injective and that $\varphi(\Omega)$ is a Runge domain w.r.t. $\Omega$.
If the sequence $(\COMPCPHI{n_l})_l$ is $\BBB$-universal, then it is universal, i.e. the operator $C_\varphi$ is hypercyclic w.r.t. $(n_l)_l$.
\end{CRL}

\begin{proof}
Let $f\in\BBB(\Omega)$ be a universal element for $(\COMPCPHI{n_l}|_{\BBB(\Omega)})_l$, and let $(K_\mu)_\mu$ be an exhaustion of $\Omega$.
For every $\mu$ there is some number $l_\mu$ such that $|f\circ\COMPPHI{n_{l_\mu}}-(1-\frac1\mu)|<\frac1\mu$ on $K_\mu$, and we may assume $l_{\mu+1}>l_\mu$.
This gives $f\circ\COMPPHI{n_{l_\mu}}\to 1$ on $\Omega$ as $\mu\to\infty$.
By Theorem \ref{the_carat_do_jedynki}, $C_\varphi$ is hereditarily hypercyclic w.r.t. $(n_{l_\mu})_\mu$ and hence hypercyclic w.r.t. $(n_l)_l$. 
\end{proof}

\subsection{Simple characterisation of hypercyclicity}\label{sect_simp_char}

Given two disjoint compact $\Omega$-convex sets, it is generally quite hard to check whether they are separable in $\Omega$ (or equivalently: whether their sum is $\Omega$-convex).
Such a problem appears in our characterisation of hypercyclicity, Theorem \ref{th_char_2}, for the sets $K$ and $\COMPPHI{n_l}(K)$.
On the other hand, in simply or infinitely connected planar domains a simpler description works: $C_\varphi$ is hypercyclic w.r.t. $(n_l)_l$ if and only if $\varphi$ is injective, $\varphi(\Omega)$ is a Runge domain w.r.t. $\Omega$ and $(\COMPPHI{n_l})_l$ is a run-away sequence, that is: if and only if the conditions \ref{pkon_u_1}, \ref{pkon_u_2}, \ref{pkon_u_3} are satisfied (see Section \ref{sect_one_dim}).
This in fact means that for such domains the condition that $K$ and $\COMPPHI{n_l}(K)$ are separable in $\Omega$ for some $l$ may be replaced by the condition that $K$ and $\COMPPHI{n_l}(K)$ are disjoint for some $l$ (not necessarily the same), as the last statement is just the run-away property of the sequence $(\COMPPHI{n_l})_l$.
In this section we give a class of Stein manifolds in which these conditions (which are necessary in general - Proposition \ref{prop_kon_u}) become sufficient for hypercyclicity.

\begin{DEF}\label{def_hyp_domain}
Let $\mathcal{S}$ denote the class of all connected Stein manifolds satisfynig the following condition: for every mapping $\varphi\in\OO(\Omega,\Omega)$ and every increasing sequence $(n_l)_l$, if the conditions \ref{pkon_u_1}, \ref{pkon_u_2}, \ref{pkon_u_3} are fulfilled, then the operator $C_\varphi$ is hypercyclic w.r.t. $(n_l)_l$.
\end{DEF}

If $\Omega\in\mathcal{S}$, then the conditions \ref{pkon_u_1}, \ref{pkon_u_2}, \ref{pkon_u_3} are sufficient for hypercyclicity.
But let us note that it is not needed to consider $\Omega$'s in which 'the conditions \ref{pkon_hu_1}, \ref{pkon_hu_2}, \ref{pkon_hu_3} are sufficient for hereditary hypercyclicity', because a connected Stein manifold $\Omega$ admits this property if and only if it belongs to $\mathcal{S}$.
This fact is a direct consequence of the relations between run-away and compactly divergent sequences: a sequence $(\COMPPHI{n_l})_l$ is run-away if and only if it has a compactly divergent subsequence, and it is compactly divergent if and only if each of its subsequences is run-away (the assumption that $\Omega$ is countable at infinity is needed here).

The following theorem presents some subclass of the class $\mathcal{S}$:

\begin{THE}\label{th_carat_conv}
If $\Omega$ is a connected Stein manifold for which there exists a point $z_0\in\Omega$ so that $$\lim_{\Omega\ni z\to\UZWIO}c_{\Omega}(z,z_0)=\infty,$$ then $\Omega$ belongs to the class $\mathcal{S}$.
\end{THE}

\noindent
The assumption $\lim_{\Omega\ni z\to\UZWIO}c_{\Omega}(z,z_0)=\infty$ means that all balls w.r.t. the Carath\'{e}odory pseudodistance are relatively compact in $\Omega$.
If it holds for some $z_0$, then it holds for every $z_0\in\Omega$ (this is an easy consequence of the triangle inequality).

Let us recall, for the readers who are not familiar with the Carath\'{e}odory pseudodistance, that the assumptions of the above theorem is fulfilled in particular by the following classes of bounded domains in $\CC^N$: convex domains, strictly pseudoconvex domains, domains for which each boundary point admits a weak peak function (i.e. for each $a\in\partial\Omega$ there is a holomorphic function $F:\Omega\to\DD$ with $\lim_{\Omega\ni z\to a}F(z)=1$), etc.
Since all these domains are Stein manifolds, in view of the above theorem they belong to $\mathcal{S}$.

For domains in $\CC$ belonging to $\mathcal{S}$ see Theorem \ref{th_char_w_c}.

\begin{proof}[Proof of Theorem \ref{th_carat_conv}]
Choose $\varphi$ and $(n_l)_l$ satisfying \ref{pkon_u_1}, \ref{pkon_u_2}, \ref{pkon_u_3}.
Since $(\COMPPHI{n_l})_l$ is run-away, passing to a subsequence we may assume that it is compactly divergent in $\OO(\Omega,\Omega)$.
We have $\COMPPHI{n_l}(z_0)\to\UZWIO$, so $c_\Omega(z_0,\COMPPHI{n_l}(z_0))\to\infty$.
Now Theorem \ref{the_carat_do_jedynki} does the job.
\end{proof}

\subsection{Hereditary hypercyclicity of $C_\varphi$}\label{sect_hu}

As it was said in Observation \ref{obs_hu_subseq}, hypercyclicity of $C_\varphi$ is equivalent to its hereditary hypercyclicity w.r.t. some increasing sequence $(n_l)_l\in\NN$.
But, as we shall see in this section, there are manifolds in which for every mapping $\varphi\in\OO(\Omega,\Omega)$ hypercyclicity of $C_\varphi$ gives its hereditary hypercyclicity, that is: the weakest of the conditions \ref{probl_u} - \ref{probl_hu} gives the strongest one.

Before we proceed, we need the notion of \textit{tautness}.
The reader can find general definition of taut manifold e.g. in \cite{abate}.
However, since we are interested only in connected Stein manifolds, for our purposes it suffices to restrict to the case when $\Omega$ is (biholomorphic to) a submanifold of some $\CC^M$.
This is by \cite[Theorem 5.3.9]{hormander}, which says that a $N$-dimensional Stein manifold $\Omega$ can be embedded in $\CC^{2N+1}$, i.e. there exists a regular injective proper holomorphic map $F:\Omega\to\CC^{2N+1}$.
In particular, $F(\Omega)$ is a submanifold of $\CC^{2N+1}$ and $F:\Omega\to F(\Omega)$ is a biholomorphism.
Given a finite-dimensional connected complex analytic manifold $\Omega'$, we endow the set $\OO(\Omega',F(\Omega))\subset\OO(\Omega',\CC^{2N+1})$ with the topology induced from $\OO(\Omega',\CC^{2N+1})$, and next we endow $\OO(\Omega',\Omega)$ with the topology carried by $F$ from $\OO(\Omega',F(\Omega))$.
One can prove that this topology does not depend on $F$.
We say that $\Omega$ is \textit{taut}, if every sequence $(f_n)_n\subset\OO(\DD,\Omega)$ is compactly divergent in $\OO(\DD,\Omega)$ or has a subsequence convergent in $\OO(\DD,\Omega)$.

Let us recall the following fact (see \cite[Theorem 2.4.3]{abate}):

\begin{THE}\label{th_abate_rel_comp}
Let $\Omega$ be a connected taut Stein manifold, and let $\varphi\in\OO(\Omega,\Omega)$.
If the sequence $(\COMPPHI{n})_n$ is not compactly divergent in $\OO(\Omega,\Omega)$, then it is relatively compact in $\OO(\Omega,\Omega)$.
\end{THE}

\noindent

We prove the following:

\begin{THE}\label{th_taut_and_conv_impl_hu}
Let a manifold $\Omega\in\mathcal{S}$ be taut, and let $\varphi\in\OO(\Omega,\Omega)$.
If the operator $C_\varphi$ is hypercyclic, then it is hereditarily hypercyclic.
\end{THE}

\noindent
Note that tautness does not imply that a manifold belongs to $\mathcal{S}$, and vice versa.
The examples are simple: $\DD_*$ is taut, but not in $\mathcal{S}$, while $\CC$ is in $\mathcal{S}$, but it is not taut (see Theorem \ref{th_char_w_c}).

\begin{proof}
Hypercyclicity of $C_\varphi$ implies that $\varphi$ satisfies the conditions \ref{pkon_u_1}, \ref{pkon_u_2}, \ref{pkon_u_3} with $n_l=l$.
The last condition implies that $(\COMPPHI{n})_n$ cannot be relatively compact in $\OO(\Omega,\Omega)$, because it has a compactly divergent subsequence.
Therefore by Theorem \ref{th_abate_rel_comp} the sequence $(\COMPPHI{n})_n$ is compactly divergent, so $\varphi$ satisfies the conditions \ref{pkon_hu_1}, \ref{pkon_hu_2}, \ref{pkon_hu_3} with $n_l=l$.
Then, as $\Omega\in\mathcal{S}$, the operator $C_\varphi$ must be hereditarily hypercyclic.
\end{proof}

Combining Theorems \ref{th_carat_conv} and \ref{th_taut_and_conv_impl_hu} we get:

\begin{THE}\label{th_carat_hu}
Let $\Omega$ be a connected Stein manifold for which there exists a point $z_0\in\Omega$ so that $$\lim_{\Omega\ni z\to\UZWIO}c_{\Omega}(z,z_0)=\infty.$$
Then $\Omega$ is taut and it belongs to $\mathcal{S}$.

In particular, if $\varphi\in\OO(\Omega,\Omega)$ is such that $C_\varphi$ if hypercyclic, then $C_\varphi$ is hereditarily hypercyclic.
\end{THE}

Example \ref{ex_conv_hu_noncarat} shows that there are taut planar domains which belong to $\mathcal{S}$, but do not satisfy the assumptions of Theorem \ref{th_carat_hu}.

Note that for all the domains listed in the paragraph below Theorem \ref{th_carat_conv}, hypercyclicity of $C_\varphi$ implies its hereditary hypercyclicity.
For domains in $\CC$ having this property, see Theorem \ref{th_char_w_c}.

\begin{proof}[Proof of Theorem \ref{th_carat_hu}]
In virtue of Theorems \ref{th_carat_conv} and \ref{th_taut_and_conv_impl_hu} it suffices to prove that $\Omega$ is taut.
Actually, this fact is probably known, but we could not find it in the literature in required form.
The proof is similar to the proof of \cite[Lemma 2.3.18]{abate}; we shall sketch it briefly.
Let $\rho$ denote the Poincar\'{e} distance in $\DD$, and $\DIAM_{c_\Omega}$, $\DIAM_{\rho}$ the diameters of sets with respect to $c_\Omega$, $\rho$, respectively.
Assume that $\Omega$ is a submanifold of $\CC^{2N+1}$ and fix a sequence $(f_n)_n\subset\OO(\DD,\Omega)$.
Suppose that it is not compactly divergent.
Passing to a subsequence if necessary, we may assume that there are compact sets $K_0\subset\DD$, $K_1\subset\Omega$ such that $f_n(K_0)\cap K_1\neq\varnothing$ for every $n$.

We are going to show that for every compact set $K\subset\DD$ the sum $\bigcup_n f_n(K)$ lies in some Carath\'{e}odory ball of $\Omega$.
Fix a point $w_1\in K_1$ and take $K$; we may assume that $K_0\subset K$.
Since $c_\Omega(f_n(z),f_n(w))\leq\rho(z,w)$ for $z,w\in K$, we see that $\DIAM_{c_\Omega}f_n(K)\leq\DIAM_{\rho}K$.
Hence, as $f_n(K)\cap K_1$ is non-empty, for $z\in K$ we have $$c_\Omega(f_n(z),w_1)\leq \DIAM_{c_\Omega}f_n(K)+\DIAM_{c_\Omega}K_1\leq \DIAM_{\rho}K+\DIAM_{c_\Omega}K_1:=R_K.$$
This implies that each $f_n(K)$ lies in the closed $c_\Omega$-ball with center at $w_1$ and radius $R_K$.

Since any $c_\Omega$-ball is relatively compact in the topology of $\Omega$, there is $\bigcup_n f_n(K)\subset\subset\Omega$ for each compact set $K\subset\DD$.
In view of the Montel theorem, $(f_n)_n$ admits a convergent subsequence and its limit belongs to $\OO(\DD,\Omega)$.	
\end{proof}

\begin{PROBLEM}
Do there exist a connected Stein manifold $\Omega$ and a holomorphic map $\varphi:\Omega\to\Omega$ such that $C_\varphi$ is hypercyclic but non-hereditarily hypercyclic?
\end{PROBLEM}

It was shown by Ansari (\cite[Theorem 1]{ansari1}, \cite[Note 3]{ansari2}; see also \cite[Theorem 8]{grosseerdmann_families}) that if a continuous linear operator $T$ on a locally convex space is hypercyclic, then for every natural number $n$ the operator $\COMP{T}{n}$ is hypercyclic, and $T$ and $\COMP{T}{n}$ have the same set of hypercyclic vectors.
This gives that if $T$ is hypercyclic, then it is hypercyclic w.r.t. every increasing arithmetic sequence of natural numbers.
Using a similar reasoning, one can show that if $T$ is hereditarily hypercyclic w.r.t. some increasing arithmetic sequence of natural numbers, then it is hereditarily hypercyclic.
Applying this conclusions to $C_\varphi$, we can state the following proposition:

\begin{PROP}\label{prop_arithm}
Let $\Omega$ be a connected Stein manifold and $\varphi\in\OO(\Omega,\Omega)$.
\begin{enumerate}
\renewcommand{\theenumi}{(\arabic{enumi})}
\renewcommand{\labelenumi}{\theenumi}
\item\label{prop_arithm_u} If $C_\varphi$ is hypercyclic, then it is hypercyclic with respect to each increasing arithmetic sequence of natural numbers.
\item\label{prop_arithm_hu} If $C_\varphi$ is hereditarily hypercyclic with respect to some increasing arithmetic sequence of natural numbers, then it is hereditarily hypercyclic.
\end{enumerate}
\end{PROP}

\subsection{The case of planar domain}\label{sect_one_dim}

For a domain $\Omega\subset\CC$, the following theorem (see \cite[Theorem 3.19]{grosseerdmann_mortini}; below we present it in a bit different form) describes all sequences $(\varphi_l)_{l\in\NN}\subset\OO(\Omega,\Omega)$ of injective maps for which the sequence $(C_{\varphi_l})_l$ of composition operators is universal:

\begin{THE}\label{th_char_w_c00}
Let $(\varphi_l)_l$ be a sequence of injective holomorphic self-maps of a domain $\Omega\subset\CC$.
Then:
\begin{enumerate}
\renewcommand{\theenumi}{(\arabic{enumi})}
\renewcommand{\labelenumi}{\theenumi}
\item\label{th_char_w_c_simple00} If $\Omega$ is simple connected, then $(C_{\varphi_l})_l$ is universal if and only if $(\varphi_l)_l$ is run-away.
\item\label{th_char_w_c_finitely00} If $\Omega$ is finitely connected but not simple connected, then $(C_{\varphi_l})_l$ is never universal.
\item\label{th_char_w_c_infinitely00} If $\Omega$ is infinitely connected, then $(C_{\varphi_l})_l$ is universal if and only if for every compact $\Omega$-convex subset $K\subset\Omega$ and for every $l_0$ there is $l\geq l_0$ such that $\varphi_l(K)$ is $\Omega$-convex and $\varphi_l(K)\cap K=\varnothing$.
\end{enumerate}
\end{THE}

Although Grosse-Erdmann and Mortini defined $\Omega$-convexity in different way (they said that a compact subset $K\subset\Omega$ of a domain $\Omega\subset\CC$ is $\Omega$-convex if every hole of $K$ contains a point of $\CC\setminus\Omega$), our definition agrees with their in dimension one.
Here by a hole of $K$ we mean a bounded connected component of $\CC\setminus K$.
Equivalence of both definitions follows from \cite[Theorems 1.3.1 and 1.3.3]{hormander}.

As a corollary from our considerations and from the above theorem applied to mappings $\varphi_l:=\COMPPHI{n_l}$ we obtain:

\begin{THE}\label{th_char_w_c}
Let $\Omega\subset\CC$ be a simply connected or an infinitely connected domain.
Then:
\begin{enumerate}
\renewcommand{\theenumi}{(\arabic{enumi})}
\renewcommand{\labelenumi}{\theenumi}
\item\label{th_char_w_c_supp} $\Omega$ belongs to $\mathcal{S}$.
\item\label{th_char_w_c_hu} If for a mapping $\varphi\in\OO(\Omega,\Omega)$ the operator $C_\varphi$ is hypercyclic, then it is hereditarily hypercyclic.
\end{enumerate}
\end{THE}

\begin{proof}
Part \ref{th_char_w_c_supp} follows directly from Theorem \ref{th_char_w_c00}.
We prove \ref{th_char_w_c_hu}.
It is known that a domain $\Omega\subset\CC$ is taut if and only if the set $\CC\setminus\Omega$ has at least two points (see \cite[Remark 3.2.3 (d)]{jarnickipflug}), so if $\Omega\neq\CC$, then $\Omega$ is taut and Theorem \ref{th_taut_and_conv_impl_hu} does the job.

It remains to consider the case $\Omega=\CC$.
Fix $\varphi$ for which $C_\varphi$ is hypercyclic.
By the Picard theorem the set $\CC\setminus\varphi(\CC)$ contains at most one point.
But since $\varphi$ is a homeomorphism on its image, there must be $\varphi(\CC)=\CC$ and so $\varphi$ is an automorphism of $\CC$.
Therefore $\varphi$ is an affine endomorphism.
Now it  suffices to use \cite[Theorem 3.1]{bernalgonzalez}: it says (in particular) that for $\varphi$ being an affine endomorphism of $\CC^N$, the composition operator $C_\varphi:\OO(\CC^N)\to\OO(\CC^N)$ is hypercyclic if and only if it is hereditarily hypercyclic.
\end{proof}

\begin{EX}\label{ex_conv_hu_noncarat}
Let $\Omega_0\subset\CC$ be an infinitely connected domain and let $a\in\Omega_0$.
Define $\Omega:=\Omega_0\setminus\lbrace a\rbrace$.
Then by Theorem \ref{th_char_w_c} the domain $\Omega$ satisfies the conclusion of Theorem \ref{th_carat_hu}, but it does not satisfy the assumption that $\lim_{\Omega\ni z\to\UZWIO}c_{\Omega}(z,z_0)=\infty$.
This follows from the classical Riemann extension theorem: there is $c_\Omega=c_{\Omega_0}$ on $\Omega\times\Omega$ and $c_\Omega(z_0,z)\to c_{\Omega_0}(z_0,a)<\infty$ as $z\to a$, for any $z_0$.
\end{EX}


%
%
%
%
%
%
%
%
%
%
%
%
%
%
%
%


\end{document}